\documentclass[11pt]{amsart}
\usepackage{amsopn}
\usepackage{amsmath,amsthm,amssymb}

\newcommand{\nc}{\newcommand}

 \nc{\aff}{\mathfrak{aff} } \nc{\bb}{\mathfrak{b} }
\nc{\cc}{\mathfrak{c} }  \nc{\dd}{\mathfrak{d} }
 \nc{\ggo}{\mathfrak{g} }
 \nc{\hh}{\mathfrak{h} }  \nc{\ii}{\mathfrak{i} }
 \nc{\jj}{\mathfrak{j} }  \nc{\kk}{\mathfrak{k} }
\nc{\mm}{\mathfrak{m} }   \nc{\nn}{\mathfrak{n} }
\nc{\pp}{\mathfrak{p} }  \nc{\rr}{\mathfrak{r} } \nc{\sg}{\mathfrak{s} }
 \nc{\sog}{\mathfrak{so} }  \nc{\spg}{\mathfrak{sp} }
 \nc{\sug}{\mathfrak{su} }  \nc{\slg}{\mathfrak{sl} }
 \nc{\tg}{\mathfrak{t} }  \nc{\uu}{\mathfrak{u} }
 \nc{\vv}{\mathfrak{v} } \nc{\ww}{\mathfrak{w} }
 \nc{\zz}{\mathfrak{z} }

 \nc{\ggob}{\overline{\mathfrak{g}}}

\nc{\glg}{\mathfrak{gl} }

\nc{\pca}{\mathcal{P}} \nc{\nca}{\mathcal{N}}

 \nc{\vp}{\varphi} \nc{\ddt}{\frac{{\rm d}}{{\rm d}t}}
 \nc{\la}{\langle} \nc{\ra}{\rangle}

 \nc{\SO}{{\sf SO}} \nc{\Spe}{{\sf Sp}} \nc{\Sl}{{\sf Sl}}
 \nc{\SU}{{\sf SU}} \nc{\Or}{{\sf O}} \nc{\U}{{\sf U}}
 \nc{\Gl}{{\sf Gl}} \nc{\Se}{{\sf S}} \nc{\Cl}{{\sf Cl}}
 \nc{\Spin}{{\sf Spin}} \nc{\Pin}{{\sf Pin}}

 \nc{\RR}{{\mathbb R}} \nc{\HH}{{\mathbb H}} \nc{\CC}{{\mathbb C}}
 \nc{\ZZ}{{\mathbb Z}} \nc{\FF}{{\mathbb F}} \nc{\NN}{{\mathbb N}}
 \nc{\GG}{{\mathbb G}} \nc{\JJ}{{\mathbb J}} \nc{\II}{{\mathbb I}}
 \nc{\KK}{{\mathbb K}} \nc{\DD}{{\mathbb D}}

 \nc{\ad}{\operatorname{ad}} \nc{\Ad}{\operatorname{Ad}}
 \nc{\coad}{\operatorname{coad}} \nc{\ct}{\operatorname{T}}
 \nc{\rank}{\operatorname{rank}} \nc{\Irr}{\operatorname{Irr}}
 \nc{\End}{\operatorname{End}} \nc{\Aut}{\operatorname{Aut}}
 \nc{\Inn}{\operatorname{Inn}} \nc{\Der}{\operatorname{Der}}
 \nc{\Dera}{\operatorname{Dera}} \nc{\Auto}{\operatorname{Auto}}
 \nc{\GL}{\operatorname{GL}}
 \nc{\SL}{\operatorname{SL}}

 \theoremstyle{plain}
 \newtheorem{thm}{Theorem}[section]
 \newtheorem{prop}[thm]{Proposition}
 \newtheorem{cor}[thm]{Corollary}
 \newtheorem{lem}[thm]{Lemma}

 \theoremstyle{remark}
 \newtheorem*{remark}{Remark}

 \newtheorem{example}[thm]{Example}

\newcommand{\mg}{\mathfrak n }
\newcommand{\mz}{\mathfrak z }

\newcommand{\mgg}{\mathfrak g }

\begin{document}

\title[Free nilpotent Lie algebras admitting ad-invariant metrics]
{Free nilpotent Lie algebras admitting ad-invariant metrics}

\author{Viviana J. del Barco}
\address{CONICET and ECEN-FCEIA, Universidad Nacional de Rosario, Pellegrini 250, 2000 Rosario, Argentina.
}

\email{V. del Barco: delbarc@fceia.unr.edu.ar}


\author{Gabriela P. Ovando}
\email{G. Ovando: gabriela@fceia.unr.edu.ar}


\thanks{Partially supported by Secyt-UNC and SCyT-UNR}
\thanks{
}



\dedicatory{}

\commby{}


\begin{abstract} In this work we find necessary and sufficient conditions for a free nilpotent
 or a free metabelian nilpotent Lie algebra to be endowed with an ad-invariant metric. For such nilpotent Lie algebras
admitting an ad-invariant metric the corresponding automorphisms groups are studied.
\end{abstract}

\maketitle

\section{Introduction} 

An ad-invariant metric on a Lie algebra $\mgg$ is a nondegenerate  symmetric bilinear form $\la \,,\, \ra$ which satisfies
\begin{equation}\label{adme}
\la [x, y],  z\ra + \la y, [x, z]\ra = 0 \qquad \mbox{ for all }x, y, z \in  \mgg.
\end{equation}

 Lie algebras endowed with  ad-invariant metrics (also called  ``metric'' or ``quadratic'') became relevant some years ago when they were useful in the formulation of some physical problems such as  the known Adler-Kostant-Symes scheme. They also constitute the basis for the construction of  bialgebras and they give rise to interesting pseudo-Riemannian geometry \cite{Co}. For instance in \cite{Ov1} a result originally due to Kostant  \cite{Kos} was revalidated for pseudo-Riemannian metrics: it states that the Lie algebra  of the isometry group of a naturally reductive pseudo-Riemannian space (in particular symmetric spaces) can be endowed with an ad-invariant metric. 

Semisimple Lie algebras are examples of Lie algebras admitting an ad-invariant metric since the Killing form is nondegenerate. In the solvable case, the Killing form is degenerate so one must search for another bilinear form with the ad-invariance property. The first investigations  concerning general Lie algebras with ad-invariant metrics appeared in \cite{FS,MR}. They get structure results proposing  a method to construct these Lie algebras recursively. This enables  a classification of nilpotent Lie algebras admitting ad-invariant metrics of dimension $\leq 7$  in \cite{FS} and a determination of the Lorentzian Lie algebras in \cite{Me}. The point is that by  this recursive method one can reach the same Lie algebra starting from two non-isomorphic Lie algebras. This fact difficulties the classification in higher dimensions.  More recently  a new proposal for the classification problem is presented in \cite{KO} and this is applied in \cite{Ka} to get the nilpotent Lie algebras with ad-invariant metrics of dimension $\leq 10$. 

However the basic question whether a non-semisimple Lie algebra admits such a metric is still opened. In the present paper we deal with this problem  in the family of free nilpotent and free metabelian nilpotent Lie algebras. 

\medskip

{\bf Theorem} \ref{t1}. {\em Let $\mg_{m,k}$ be the free $k$-step nilpotent Lie algebra in $m$ generators. Then $\mg_{m,k}$ admits an ad-invariant metric if and only if $(m,k)=(3,2)$ or $(m,k)=(2,3)$.}

\medskip

The techniques for the proof do not make use of the extension procedures mentioned before, but  properties of free nilpotent Lie algebras which combined with the ad-invariance condition enable the deduction of the Lie algebras  $\mg_{2,3}$ and $\mg_{3,2}$. We note that for $k=2,3$ the free and free metabelian $k$-step nilpotent Lie algebras coincide. 
For the free metabelian case the first approach lies in  the fact that 2-step solvable Lie algebras admitting ad-invariant metrics are nilpotent and at most 3-step (Lemma \ref{le11}). Thus  working out  we get the next result. 

\medskip

{\bf Theorem} \ref{t2}. {\em Let $\tilde{\mg}_{m,k}$ be the free metabelian $k$-step nilpotent Lie algebra in $m$ generators. Then $\tilde{\mg}_{m,k}$ admits an ad-invariant metric if and only if $(m,k)=(3,2)$ or $(m,k)=(2,3)$.}

\medskip

These two Lie algebras  have been studied since a long time in sub-Riemannian geometry \cite{Mo}. 
Thus $\mg_{2,3}$ is associated to the Carnot group distribution (see for instance \cite{BM}), which is related to the ``rolling balls problem'', treated by Cartan in \cite{Ca}. The prolongation, representing -roughly speaking- the maximal possible symmetry of the  distribution, in the case of  $\mg_{2,3}$ is the exceptional Lie algebra $\mgg_2$  \cite{BM}. The Lie algebra  $\mg_{3,2}$ was studied more recently in \cite{My} in the context of the geometric characterization of the so-called Maxwell set, wave fronts and caustics (see also \cite{MA}). 

We complete the work with a study of the group of  automorphisms of the Lie algebras  $\mg_{2,3}$ and $\mg_{3,2}$.  The corresponding structure is described and in particular the subgroup of orthogonal automorphisms is determined. 

Following \cite{FS} and the considerations above all Lie algebras here are over a field $K$ of characteristic 0, nevertheless some results in Section 3 could be still true for fields of a characteristic different from 2. 

\section{Free and free metabelian nilpotent Lie algebras}

Let $\mgg$ denote a  Lie algebra. The so-called central descending and ascending series of  $\mgg$, respectively $\{C^r(\mgg)\}$ and $\{C_r(\mgg)\}$ for all $r\geq 0$, are constituted by the ideals in $\mgg$, which  for non-negative integers $r$,  are given by
$$\begin{array}{rclrcl}
C^0(\mgg)&=&\mgg & C_0(\mgg)&=&0 \\
C^r(\mgg)&=&[\mgg,C^{r-1}(\mgg)] & C_r(\mgg)&=&\{x\in \mgg:[x,\mgg]\in C_{r-1}(\mgg)\}.
\end{array}
$$

Note that $C_1(\mgg)$ is by definition the center of $\mgg$, which will be denoted by $\mz(\mgg)$.

A Lie algebra $\mgg$ is called \emph{k-step nilpotent} if $C^k(\mgg)=\{0\}$ but $C^{k-1}(\mgg)\neq \{0\}$ and clearly $C^{k-1}(\mgg)\subseteq \mz(\mgg)$. 

\begin{example} Heisenberg algebras. Let $X_1, \hdots, X_n, Y_1, \hdots, Y_n$ denote a basis of the $2n$-dimensional real vector space $V$ and let $Z\notin V$. Define
$[X_i, Y_j]=\delta_{ij} Z$ and $[Z, U]=0$ for all $U\in V$. Thus $\hh_n= V \oplus \RR Z$ is  the Heisenberg Lie algebra of dimension $2n+1$, which is  2-step nilpotent.
\end{example}

We shall make use of the notation $\mgg'=[\mgg, \mgg]$ and $\mgg''=[\mgg',\mgg']$. 
A Lie algebra is called {\em 2-step solvable} if its commutator is abelian, that is  $\mgg''=0$.

Let  $\mathfrak f_m$ denotes the free  Lie algebra on $m$ generators, with $m\geq 2$. (Notice that a unique  element spans an abelian Lie algebra). Thus
\begin{itemize}

\item the free metabelian $k$-step nilpotent Lie algebra on $m$ generators is defined as $$\tilde{\mg}_{m,k}:= \mathfrak f_m/(C^{k+1}(\mathfrak f_m)+ \mathfrak f_m''),$$

\item the \emph{free $k$-step nilpotent} Lie algebra on $m$ generators $\mg_{m,k}$ is defined as the quotient algebra 
$$\mg_{m,k}:=\mathfrak f_m/C^{k+1}(\mathfrak f_m).$$ 
\end{itemize}

In particular free metabelian nilpotent of any degree are 2-step solvable.

\begin{remark} \label{23} For $k=2,3$ any $k$-step nilpotent Lie algebra is 2-step solvable, which follows from the Jacobi identity. Thus for the free nilpotent ones  we get $\tilde{\mg}_{m,k}=\mg_{m,k}$ for $k=2,3$.
\end{remark}

Let $\mg_{m,k}$ be a free $k$-step nilpotent Lie algebra and let $\{e_1,\ldots,e_m\}$ be an ordered set of generators. The construction of a \emph{Hall} basis associated to this sets of generators is explained below (see \cite{Ha,GG}). 

Start by defining the \emph{length} $\ell$ of each  generator as $1$. Take the Lie brackets $[e_i, e_j]$ for $i>j$, which by definition satisfy $\ell([e_i,e_j])=2$. Now the elements $e_1, \dots , e_m$, $[e_i, e_j]$, $i>j$ belong to the Hall basis. Define a total order in that set by extending the order of the set of  generators and so that $E>F$ if $\ell(E)>\ell(F)$. They allow the construction of the elements of length $3$ and so on. 

Recursively each element of the Hall basis of $\mg_{m,k}$ is defined as follows. The generators $e_1,\ldots,e_m$ are elements of the basis  of length 1. Assume we have defined basic elements of lengths $1,\ldots, r-1\leq k-1$, with a total order satisfying  $E > F$ if $\ell(E) > \ell(F)$.

If $\ell(E) =s$ and $\ell(F) = t$ and $r = s + t\leq k$, then $[E, F]$ is a basic element of length $r$ if both of the following conditions hold:

\begin{itemize}
\item[(1)] $E$ and $F$ are basis elements and $E>F$, and
\item[(2)] if $\ell(E)>1$ and $E=[G,H]$ is the unique decomposition with $G,H$ basic elements, then $F\geq H$.
\end{itemize}

This gives rise to a natural graduation of $\mg_{m,k}$:
$$\mg_{m,k}=\bigoplus_{s=1}^k \mathfrak p(m,s),$$
where $\mathfrak p(m,s)$ denotes the subspace spanned by the elements of the Hall basis of length $s$. Notice that 
\begin{itemize}
\item $C^r(\mg_{m,k})=\oplus_{s=r+1}^k\mathfrak p(m,s)$,
\item $\mathfrak p(m,k) = \mz(\mg_{m,k})$
\end{itemize}

The first assertion follows from the fact that every  bracket of $r + 1$ elements of $\mg_{m,k}$, is a linear combination of brackets of $r + 1$ elements in
the Hall basis (see proof of Theorem 3.1 in \cite{Ha}). This implies $C^r(\mg_{m,k})\subseteq \oplus_{s=r+1}^k \mathfrak p(m, s)$; the
other inclusion is obvious. In particular, $\mathfrak p(m, k) = C^{k-1}(\mg_{m,k}) \subseteq \mathfrak z(\mg_{m,k})$.  Now let $x\in \mathfrak z(\mg_{m,k})$ and let $e$ be a generator and assume $x\notin C^{k-1}(\mg_{m,k})$. Recall that $\mg_{m,k}$ is homomorphic image of the free Lie algebra $\mathfrak f_m$ so that there exist $X,E\in \mathfrak f_m$ such that $X\to x$ and $E \to e$, being $E$  a generator of $\mathfrak f_m$. Since $[x, e]=0$ then $[X,E]=0$ which says $X$and $E$ are proportional, which is impossible (see for instance Ch. 2 in \cite{Ba}). Thus $\mathfrak p(m, k) =  \mathfrak z(\mg_{m,k})$.

\medskip 

Denote as $d_m(s)$ the dimension of $\mathfrak p(m,s)$. Inductively one gets \cite{Se}
\begin{equation}\label{dim}s\cdot d_m(s)= m^s- \sum_{r|s, r<s} r\cdot d_m(r),\qquad s\geq 1.
\end{equation}

Hence for a fixed $m$, one has  $d_m(1)=m$ and $d_m(2)=m(m-1)/2$. 

\begin{example}
\label{ex:1}Given an ordered set of generators $e_1,\ldots, e_m$ of a free $2$-step nilpotent  Lie algebra $\mg_{m,2}$, a Hall basis is
\begin{equation}\label{ba2}
	\mathcal B=\{e_i,\,[e_j,e_k]:\,i=1,\ldots,m,\, 1\leq  k< j\leq m\}. 
	\end{equation}
	Equation (\ref{dim}) asserts that $\dim\mg_{m,2}=d_m(1)+d_m(2)=m+m(m-1)/2$. 
Since $\mathfrak z(\mg_{m,2}) = \mathfrak p(m, 2)$, we have $\dim \mathfrak z(\mg_{m,2})= m(m-1)/2$.
\end{example}

\begin{example}
	 \label{ex:2} For the free $3$-step nilpotent Lie algebra on $m$ generators $\mg_{m,3}$ a Hall basis of a set of generators as before has the form
	\begin{equation} \label{ba3}
	\mathcal B=\{e_i, \,[e_j,e_k],\,[[e_r,e_s],e_t]:\,i=1,\ldots,m,\,1\leq k< j\leq m,\, 1\leq s<r\leq m,\,t\geq s\}. 
	\end{equation}
	
	It holds $\mz(\mg_{m,3})=\mathfrak p(m,3)$, and so 
$$\dim\, \mz(\mg_{m,3})= d_m(3)=m(m^2-1)/3.$$
\end{example}

\section{Free and free metabelian nilpotent Lie algebras and ad-invariant metrics}

In this section we determine free nilpotent and free metabelian nilpotent Lie algebras admitting ad-invariant metrics.


\vskip 3pt

Let $\mgg$ denote a Lie algebra equipped with an ad-invariant metric $\la\,,\,\ra$, see (\ref{adme}). If $\mm \subseteq \ggo$ is a subset, then we denote by $\mm^{\perp}$ the linear subspace of $\mgg$ given by
$$\mm^{\perp}=\{x\in \ggo, \la x, v\ra=0 \mbox{ for all } v\in \mm\}.$$
 In particular $\mm$ is called 
\begin{itemize}
\item {\em isotropic} if $\mm \subseteq \mm^{\perp}$, 
\item {\em totally isotropic} if $\mm=\mm^{\perp}$, and 
\item {\em nondegenerate} if and only if $\mm \cap \mm^{\perp}=\{0\}$.
\end{itemize}

The proof of the next result follows easily from  an inductive procedure.

\begin{lem} \label{le1} Let $(\ggo, \la\,,\,\ra)$ denote a Lie algebra equipped with an ad-invariant metric.
\begin{itemize}
\item[(1)] If $\hh$ is an ideal of $\ggo$ then $\hh^{\perp}$ is also an ideal in $\ggo$.
\item[(2)] $C^r(\ggo)^{\perp}=C_r(\ggo)$ for all $r$.
\end{itemize}
\end{lem}

Thus on any Lie algebra admitting an ad-invariant metric the next equality holds
\begin{equation}\label{e2}
\dim \ggo=\dim C^r(\ggo) + \dim C_r(\ggo).
\end{equation}
For the case $r=1$ one obtains
\begin{equation}\label{e1}
\dim \ggo=\dim \zz(\ggo) + \dim C^1(\ggo).
\end{equation}

\begin{example} Let $\mg$ denote a $2$-step nilpotent Lie algebra equipped with an ad-invariant metric. Assume $\zz(\mg)=C^1(\mg)$, then by (\ref{e1}) the metric is neutral and $\dim \mg = 2 \dim \zz(\mg)$.  As a consequence the Heisenberg Lie algebra $\hh_n$ cannot be equipped with any ad-invariant metric. 

 Examples of nilpotent Lie algebras satisfying the equality (\ref{e2}) above for every $r$ arise by considering  the semidirect product of a nilpotent Lie algebra $\nn$ with its dual space via de coadjoint representation $\nn \ltimes \nn^*$. The natural neutral metric on $\nn \ltimes \nn^*$ is ad-invariant.

Nevertheless, condition (\ref{e1}) (and hence (\ref{e2})) is not sufficient for a 2-step nilpotent Lie algebra to admit an ad-invariant metric as shown for instance in \cite{Ov2}.

\end{example}

\begin{lem} \label{le11} Let $\mgg$ denote a 2-step solvable Lie algebra provided with an ad-invariant metric, then $\mgg$ is nilpotent and at most 3-step.
\end{lem}
\begin{proof} Let $\la\,,\,\ra$ denote an ad-invariant metric on $\mgg$. Since $\mgg$ is 2-step solvable for all $x,y \in \mgg'$ one has $[x,y]=0$, which is equivalent to
$$\begin{array}{rcll} 
0 & = & \la [x,y], u\ra \qquad & \mbox{ for all } u \in \mgg\\
 & = & \la [u, x], y\ra \qquad & \mbox{ for all } y\in \mgg'
 \end{array}
 $$
 thus $[u,x]\in [\mgg, \mgg]^{\perp}=\zz(\mgg)$, and since $x\in \mgg'$ can be written as $x=[v,w]$, then $[u,[v,w]]\subseteq \zz(\mgg)$ for all $u,v,w\in \mgg$, that is $C^4(\mgg)=0$ and so $\mgg$ is at most 3-step nilpotent.
 \end{proof}

\begin{cor} \label{corofree} Let $\tilde{\mg}_{m,k}$ denote a free metabelian nilpotent Lie algebra admitting an ad-invariant metric, then $k\leq 3$.
\end{cor}

Remark \ref{23} and the previous result says that a free metabelian nilpotent Lie algebra with an ad-invariant metric if it exists, is free nilpotent. Below we determine  which free nilpotent Lie algebra admits such a metric.
\smallskip

Whenever $\mg_{m,k}$ is free nilpotent we have that $\dim \mg_{m,k}/C^1(\mg_{m,k}) =m$ so that
\begin{equation}
\label{e3}
\dim \mg_{m,k}=m+\dim C^1(\mg).
\end{equation}

Hence Equations (\ref{e1}) and (\ref{e3}) show that if $\mg_{m,k}$ admits an ad-invariant metric then 
\begin{equation}
\label{e4}
\dim \mz(\mg_{m,k})=m.
\end{equation}

\begin{prop}\label{p1} If $\nn_{m,2}$ is a free 2-step nilpotent Lie algebra endowed with an ad-invariant metric, then $m=3$.
\end{prop}
\begin{proof}
 Let $\mg_{m,2}$ be the free 2-step nilpotent Lie algebra on $m$ generators. 
As we showed in  Example \ref{ex:1} its center has dimension $m(m-1)/2$. Now if Equation (\ref{e4}) holds then  $m=3$.
\end{proof}

\begin{prop}\label{p2} Let $\nn_{m,3}$ be a free 3-step nilpotent Lie algebra provided with an ad-invariant metric, then $m=2$.
\end{prop}
\begin{proof}
As shown in Example \ref{ex:2} the center $\mz(\mg_{m,3})$ has dimension $d_m(3)=m(m^2-1)/3$.
From straightforward calculations, if Equation (\ref{e4}) is satisfied then  $m= 2$. 
\end{proof}

\begin{prop}\label{p3} No free $k$-step nilpotent Lie algebra $\mg_{m,k}$ on $m$ generators with $k \geq 4$ can be endowed with an ad-invariant metric.
\end{prop}
\begin{proof}
\emph{$\bullet$ 4-step nilpotent case:} In this case $\mathfrak p(m,4)=  \mz(\mg_{m,4})$, thus from (\ref{dim}):
$$\dim\mz(\mg_{m,4})\geq d_m(4)=\frac{1}{4}(m^4-d_m(1)-2d_m(2))= \frac{m^2(m^2-1)}{4}.$$ 
 Notice that for $m\geq 2$, one has $m^2(m^2-1)/4>m$. 
\medskip 

\emph{$\bullet$ General case, $k\geq 5$:} Let $\mg_{m,k}$ denote the free $k$-step nilpotent Lie algebra in $m$ generators. The goal here is to show that for every $m$ and $k\geq 5$ the dimension of the center of $\mg_{m,k}$ is greater than $m$. In order to give a lower bound for $\dim\mz(\mg_{m,k})$ we construct elements of length $k$ in a Hall basis $\mathcal B$.

Let $\{e_1,\ldots,e_m\}$ a set of generators of $\mg_{m,k}$ and consider the set
\begin{eqnarray}
\mathcal U &=&\{[[[e_i,e_j],e_k],e_m]: 1\leq j<i\leq m,\,k\geq j\}.\nonumber  
\end{eqnarray}

Any element in $\mathcal U$ is basic and of length 4. Given $x\in \mathcal U$, the bracket $$[x,e_m]^{(s)}:= [[[x,\overbrace{e_m],e_m]\cdots , e_m]}^{s}  \;\;s\geq 1$$
is an element in the Hall basis if $\ell([x,e_m]^{(s)})\leq k$. 

In fact if $s=1$ then 
\begin{itemize}
\item[(1)] both $x=[[[e_i,e_j],e_k],e_m] \in \mathcal U$ and $e_m$ are elements of the Hall basis, and $x >e_m$ because of their length;
\item[(2)] also $x=[G,H]$ with $G=[[e_i,e_j],e_k]$ and $H=e_m$ and we have $e_m\geq H$. 
\end{itemize}

So both conditions of the Hall basis definition are satisfied and hence 
 $[x,e_m]^{(1)}\in \mathcal B$ and it belongs to $ C^4(\mg_{m,k})$.
 
 Inductively suppose  $[x,e_m]^{(s-1)}\in\mathcal B$, then clearly $[[[x,e_m],e_m]\cdots ], e_m]^{(s-1)}>e_m$ and it is possible to write $[x,e_m]^{(s-1)}=[G,H]$ with $H=e_m$. Thus $[x,e_m]^{(s)}\in\mathcal B$.  Notice that $[x,e_m]^{(s)}\in C^{s+3}(\mg_{m,k})$.

We construct the new set 
$$ \widetilde{\mathcal U}:=\{[x,e_m]^{(k-4)}:x\in \mathcal U\}\subseteq C^{k-1}(\mg_{m,k}),$$
which is contained in the center of $\mg_{m,k}$ and  it is linearly independent. Therefore
\begin{equation}\label{u}
\dim \mz(\mg_{m,k}) \geq |\widetilde{\mathcal U}|. 
\end{equation}

Clearly $\widetilde{\mathcal U}$ and $\mathcal U$ have the same cardinal. Also, $|\mathcal U|=\sum_{j=1}^m(m-j+1)(m-j)$ since for every fixed $j=1,\ldots,m$,  the amount of possibilities to choose $k\geq j$ and $i>j$ is $(m-j+1)$ and $(m-j)$  respectively. 

Straightforward computations give $ |\widetilde{\mathcal U}|=1/3\, m^3+m^2+2/3\,m$ which combined with (\ref{u}) proves that for any $m$ and $k\geq 5$ $$\dim \mz(\mg_{m,k})\geq 1/3 \,m^3+m^2+2/3\,m .$$ The right hand side is greater than $m$ for all $m\geq 2$. According to (\ref{e4}), the free $k$-step nilpotent Lie algebra $\mg_{m,k}$ does not admit an ad-invariant metric if $k\geq 5$.
\end{proof}

\begin{thm} \label{t1} Let $\mg_{m,k}$ be the free $k$-step nilpotent Lie algebra in $m$ generators. Then $\mg_{m,k}$ admits an ad-invariant metric if and only if $(m,k)=(3,2)$ or $(m,k)=(2,3)$.
\end{thm}
\begin{proof} Propositions (\ref{p1}) (\ref{p2}) and (\ref{p3})  prove that if $\mg_{m,k}$ admits an ad-invariant metric then $(m,k)=(3,2)$ or $(m,k)=(2,3)$. Let us show the resting part of the proof.

The Lie algebra $\mg_{3,2}$ has a basis $\{e_1,e_2,e_3,e_4,e_5,e_6\}$ with non zero brackets 
\begin{equation}\label{b32}
[e_1,e_2]=e_4,\qquad [e_1,e_3]=e_5,\qquad [e_2,e_3]=e_6. 
\end{equation}
Since $C^1(\mg_{3,2})=\zz(\mg_{3,2})$, if the metric $\la \,,\,\ra$ is ad-invariant the center of $\mg_{3,2}$ is totally isotropic, so it must hold
$$\la e_i, e_j\ra=0 \qquad \mbox { for all }i,j=4,5,6.$$

The ad-invariance property says
$$\la e_4, e_1\ra=\la [e_1,e_2], e_1\ra=0$$
and similarly $\la e_4,e_2\ra=0$, therefore $\la e_4,e_3\ra \neq 0$. 
Analogously, $\la e_5, e_2\ra\neq 0$ and $\la e_6, e_1\ra\neq 0$. Moreover
$$\alpha=\la [e_1, e_2], e_3\ra=-\la e_2, e_5, \ra=\la e_1, e_6\ra.$$

Thus  in the ordered basis $\{e_1,e_2,e_3,e_4,e_5,e_6\}$, a matrix  of the form
\begin{equation}\label{m32}
\left(
\begin{matrix}
a_{11} & a_{12} & a_{13} & 0 & 0 & \alpha \\
a_{12} & a_{22} & a_{23} & 0 & -\alpha & 0 \\
a_{13} & a_{23} & a_{33} & \alpha & 0 & 0 \\
 0 & 0  & \alpha & 0 & 0 & 0 \\
 0 & -\alpha &  0 & 0 & 0 & 0 \\
 \alpha & 0 & 0 & 0 & 0 & 0
 \end{matrix}
 \right) \qquad \mbox{ with } a_{ij} \in K, \forall i,j=1,2,3 \mbox{ and } \alpha \neq 0
 \end{equation}
 
corresponds to an ad-invariant metric on $\mg_{3,2}$.
\medskip
 
The Lie algebra $\mg_{2,3}$ has a basis $\{e_1,e_2,e_3,e_4,e_5\}$ with non-zero Lie brackets 
\begin{equation}\label{b23}
[e_1,e_2]=e_3,\qquad [e_1,e_3]=e_4,\qquad [e_2,e_3]=e_5. 
\end{equation}
Let $\la \,,\,\ra$ be an ad-invariant metric on $\mg_{2,3}$. Then $\zz(\mg_{3,2})=C^2(\mg_{2,3})=span\{e_4, e_5\}$ while $C^1(\mg_{2,3})=C^2(\mg_{2,3}) \oplus K e_3$.

The ad-invariance property also says that
$$0=\la e_4, e_3\ra =\la e_4, e_4\ra= \la e_4, e_5\ra=\la e_5, e_3\ra=\la e_5, e_5\ra.$$ 
Moreover
$$\la e_1, e_3\ra= \la e_1, [e_1, e_2]\ra =0 \quad \mbox { and } \quad \la e_2, e_3\ra=\la e_2, [e_1, e_2] \ra=0;$$
$$\la e_1, e_4\ra=\la e_1, [e_1, e_3]\ra=0 \quad \mbox{ and } \quad \la e_2, e_5\ra=\la e_2, [e_2, e_3]\ra=0.$$
Therefore
$$\la e_3, e_3\ra=\la [e_1,e_2], e_3\ra=-\la e_2, e_4\ra= \la e_1, e_5\ra =\alpha \neq 0,$$
which amounts to the following matrix for $\la \,,\,\ra$ in the ordered basis above:
\begin{equation}\label{m23}
\left(
\begin{matrix}
a_{11} & a_{12} & 0 & 0 & \alpha \\
a_{12} & a_{22} & 0 & - \alpha & 0 \\
0 & 0 & \alpha & 0 & 0 \\
0 & -\alpha & 0 & 0 & 0 \\
\alpha & 0 & 0 & 0 & 0 
\end{matrix}
\right)
\qquad \mbox{ with } a_{ij} \in K, \forall i,j=1,2 \mbox{ and } \alpha \neq 0. 
\end{equation}
\end{proof}

 Remark \ref{23},  Corollary \ref{corofree}  and the previous theorem imply the next result. 

\begin{thm}\label{t2} Let $\tilde{\mg}_{m,k}$ be the free metabelian $k$-step nilpotent Lie algebra in $m$ generators. Then $\tilde{\mg}_{m,k}$ admits an ad-invariant metric if and only if $(m,k)=(3,2)$ or $(m,k)=(2,3)$.
\end{thm}

\begin{remark} The free nilpotent Lie algebras above can be constructed as extensions of abelian
Lie algebras. This is the way in which they appear in \cite{FS}, where Favre and Santharoubarne
obtained the classification of the nilpotent Lie algebras of dimension $\leq  7$ admitting an ad-invariant
metric. According to their results, any of the Lie algebras equipped with an
ad-invariant metric $\mg_{3,2}$ or $\mg_{2,3}$ as above is equivalent to one of the followings

\begin{equation}\label{m3223}
(\mg_{3,2}, B_{3,2}) : \left(
\begin{matrix}
0 & 0 & 0 & 0 & 0 & 1 \\
0 & 0 & 0 & 0 & -1 & 0 \\
0 & 0 & 0 & 1 & 0 & 0 \\
 0 & 0  & 1 & 0 & 0 & 0 \\
 0 & -1 &  0 & 0 & 0 & 0 \\
 1 & 0 & 0 & 0 & 0 & 0
 \end{matrix}
 \right) \;\;
 (\mg_{2,3}, B_{2,3}) : 
 \left(
\begin{matrix}
0 & 0 & 0 & 0 & 1 \\
0 & 0 & 0 & - 1 & 0 \\
0 & 0 & 1 & 0 & 0 \\
0 & -1 & 0 & 0 & 0 \\
1 & 0 & 0 & 0 & 0 
\end{matrix}
\right).
\end{equation}

\end{remark}

\section{The automorphism groups of $\mg_{3,2}$ and $\mg_{2,3}$}
Here we study the automorfisms of the Lie algebras in Theorem \ref{t1}. This is indeed  a topic of active research  (see for instance \cite{DF,DG} and references therein). Our goal is to  write explicitly the  algebraic structure, in terms of the actions and representations of the different subgroups or subalgebras. We also distinguish the subgroup of orthogonal automorphism (resp. the Lie algebra of skew-symmetric derivations) in presence of the ad-invariant metric fixed in (\ref{m3223}).

\medskip

Recall that a derivation of a Lie algebra $\mathfrak g$ is a linear map $t: \mathfrak g \to \mathfrak g$ satisfying
$$t[x,y]=[tx, y] + [x, ty ] \qquad \mbox{ for all } x,y \in \mathfrak g.$$

Whenever $\mathfrak g$ is endowed with a metric $\la \,,\,\ra$ a skew-symmetric derivation of $\mgg$ is a derivation  $t$ such that 
\begin{equation}	\label{antisim}
\la t a, b\ra=-\la a, tb\ra \quad \text{ for all }\; a,b\in\mgg.
\end{equation}

We denote by $\Der(\mgg)$  the Lie algebra of derivations of $\mgg$, which is the Lie algebra of the group of automorphisms of $\ggo$, $\Aut(\mgg)$. 
Let $\Dera(\mgg)$ denote the subalgebra of  $\Der(\mgg)$ consisting of skew-symmetric derivations of $(\mgg, \la\,,\,\ra)$.  Thus $\Dera(\mgg)$ is the Lie algebra of the group of orthogonal automorphisms denoted  by $\Auto(\mgg)$:
$$\Auto(\mgg)=\{ \alpha \in \Aut(\mgg) : \la \alpha x, \alpha y \ra=\la x,y \ra \quad \forall x,y \in \mgg\}.
$$

For an arbitrary Lie algebra $\mgg$, each $t \in \Aut(\mgg)$ leaves invariant both the commutator ideal $C^1(\mgg)$ and the center $\zz(\mgg)$. Thus $t$ induces automorphisms of  the quotient Lie algebras $\mgg/\zz(\mgg)$ and $\mgg/C^1(\mgg)$.

In particular if $\mgg$ is solvable, $\mgg/C^1(\mgg)$ is a nontrivial abelian Lie algebra. Hence any  $t\in\Aut(\mgg)$ induces an element $s \in \Aut(\mgg/C^1(\mgg))$ which as an automorphism of an abelian Lie algebra, $s \in \GL(p,K)$ where $p = \dim( \mgg/C^1(\mgg))$. Note that if $\mgg$ is free nilpotent, $p$ coincides with the number of generators of $\mgg$.

Below we proceed to the study in each case.

\subsection{The Lie algebra $\mg_{2,3}$} 

Let $t$ denote an automorphism of $\mg_{2,3}$
and let $e_1,e_2,e_3,e_4,e_5$ be the basis given in (\ref{b23}). Denote as $(t_{ij})$ the matrix of $t$ in that basis.
 Since $t$ leaves invariant the commutator and the center and $C^1(\mg_{2,3})=span\,\{e_3,e_4,e_5\}$ and $\mz(\mg_{2,3})=span\,\{e_4,e_5\}$, we have that $t_{ij}=0$ if $i=1,2$,  $j=3,4,5$.

Notice that the Lie algebra $\mg_{2,3}/\zz(\mg_{2,3})$ is isomorphic to the Heisenberg Lie algebra $\hh_1$, hence $t$ induces a automorphism $\bar{t}\in \Aut(\hh_1)$. Let $\bar{x}$ denote the image of an arbitrary element $x\in \mg_{2,3}$ by the canonical epimorphism $\mg_{2,3} \to \mg_{2,3}/\zz(\mg_{2,3})$, thus $\bar{t}(\bar{x})=\overline{tx}$.
From the computation 
$[\bar{t} \bar{e}_1, \bar{t}\bar{e}_2]=\overline{te_3}$ one gets
\begin{equation}\label{det}
t_{11} t_{22}- t_{12} t_{21}= t_{33}.
\end{equation}

We introduce the notation for the submatrices 
$$A:=\left(\begin{array}{cc}t_{11}&t_{12}\\ t_{21}& t_{22}\end{array}\right)\qquad B:=\left(\begin{array}{cc}t_{44}&t_{45}\\ t_{54}& t_{55}\end{array}\right)
.$$
 Note that the second matrix is non-singular since $t$ is non-singular.

By compairing (\ref{det}) with the computations 
$[t e_i, t e_j]= t[e_i,e_j]$ we get  the conditions
$$t_{33}=\det A,\qquad B=\det(A) A, \qquad 
\left(\begin{array}{c}t_{43}\\ t_{53}\end{array}\right)=A \left(\begin{array}{c}t_{32}\\ -t_{31}\end{array}\right).$$

So for any $t\in \Aut(\mg_{2,3})$, $(t_{ij})_{i,j}$  has the form
\begin{equation}\label{t}
\left(\begin{array}{ccc}
A &                                \begin{array}{cc}0\\0\end{array} & \begin{array}{cc}0&0\\0&0\end{array} \\
\begin{array}{cc} t_{31}&t_{32}\end{array} &      \det(A)                               & \begin{array}{cc}0&0\end{array} \\
\begin{array}{cc}t_{41}&t_{42}\\t_{51}&t_{52}\end{array} &\begin{array}{cc}t_{11}t_{32}-t_{31}t_{12}\\t_{21}t_{32}-t_{31}t_{22}\end{array} & \det(A) A\end{array}
\right), 
\end{equation}
where $A\in \GL(2,K)$.

Consider the following matrices
  in $\Aut(\mg_{2,3})$:
$$\mathcal G=\left\{ \tilde{A}=\left(\begin{array}{ccc}
A &                                \begin{array}{cc}0\\0\end{array} & \begin{array}{cc}0&0\\0&0\end{array} \\
\begin{array}{cc} 0&0\end{array} &        \det(A)                                & \begin{array}{cc}0&0\end{array} \\
\begin{array}{cc}0&0\\0&0\end{array} &\begin{array}{cc}0\\0\end{array} &\det(A)A\end{array}
\right), \, A\in \GL(2,K)\right\}$$
$$
\mathcal H=\left\{ h_{(x,y,z)}=\left(
\begin{array}{ccccc}
 1&0&0&0&0 \\
0&1&0&0&0\\
y&x&1&0&0\\ 
z+\frac{1}{2} x y&\frac{1}{2} x^2&x&1&0\\
-\frac{1}{2} y^2&z-\frac{1}{2} xy&-y&0&1\end{array}\right),\, (x,y,z)\in K^3\right\}
$$

$$\mathcal R= \left\{ r_{(u,v,w)}=\left(
\begin{array}{ccccc}
 1&0&0&0&0 \\
0&1&0&0&0\\
0&0&1&0&0\\ 
v&w&0&1&0\\
u&-v&0&0&1\end{array}\right),\, (u,v,w)\in K^3\right\}.$$

Note that $\mathcal{G}$, $\mathcal{R}$ and  $\mathcal{H}$ are subgroups of $\Aut(\mg_{2,3})$. 

Moreover every $t$ of the form  (\ref{t}) can be written as a product of matrices $$t=\tilde{A}\cdot r_{(u,v,w)} \cdot h_{(x,y,z)},\qquad \mbox{ with } \tilde{A} \in \mathcal G,\; r_{(u,v,w)}\in \mathcal R,\; h_{(x,y,z)}\in \mathcal H$$

 where $x=t_{32}/\det(A)$, $y=t_{31}/\det(A)$, $z=(t_{22}t_{41}-t_{12}t_{51}+t_{52}t_{11}-t_{42}t_{21})/2\det(A)^2$ and
 
  $u=(2t_{51}t_{11}-2t_{41}t_{21}+t_{31}^2)/2 \det A^2$, $v=(t_{22}t_{41}-t_{31}t_{32}-t_{12}t_{51}+t_{42}t_{21}-t_{52}t_{11})/2 \det A$,  $w=(2t_{42}t_{22}-t_{32}^2-2t_{52}t_{12})/2\det A^2$.

 The elements of $\mathcal{H}$ commute with those of $\mathcal{R}$; also $\mathcal H \cap  \mathcal R=\{I\}$. 
 Hence $\mathcal{R}\cdot \mathcal{H}\simeq \mathcal{R}\times \mathcal{H}.$ 
 
 It holds $h_{(x,y, z)} \cdot h_{(x',y',z')}=h_{(x+x', y+ y', z + z' + \frac12 (xy' -x' y))}$, where $\cdot$ denotes the product of matrices. Thus the map from the Heisenberg Lie group $H_1$ to $\Aut(\mg_{2,3})$ given by $(x,y, z) \mapsto h_{(x,y,z)}$ is an isomorphism of groups.

Analogously $(u,v,w)\mapsto r_{(u,v,w)}$ is an isomorphism of groups from $K^3$ to $ \mathcal R\subseteq \Aut(\mg_{2,3})$.

The action by conjugation of $\mathcal{G}$ preserves both $\mathcal{H}$ and $\mathcal{R}$. Thus  the map 
$$\tau_{\tilde{A}}(r,h)=(\tilde{A} r\tilde{A}^{-1},\tilde{A} h\tilde{A}^{-1}).$$  
defines a group homomorphism from $\mathcal{G}$ to $\Aut(\mathcal R \times \mathcal{H})$. 

The subgroup $\mathcal{R}\times \mathcal{H}$ is normal in $\Aut(\mg_{2,3})$ and $\mathcal{G}\cap (\mathcal{R}\times \mathcal{H})=\{I\}$, hence $\Aut(\mg_{2,3})\simeq \mathcal G\ltimes_\tau (\mathcal{R}\times \mathcal{H})$ (\cite{Kn}).
It is clear that $\mathcal{G}\simeq \GL(2,K)$, $\mathcal{R}\simeq K^3$ and $\mathcal{H}\simeq H_1$ and these isomorphisms preserve the action of $\GL(2,K)$ in $H_1$ and $K^3$. So the next result follows.

\begin{prop}\label{pro5} The group of automorphisms of $\mg_{2,3}$ is 
$$\Aut(\mg_{2,3}) \simeq \GL(2,K)\ltimes (K^3 \times H_1),$$
 where $H_1$ denotes the Heisenberg Lie group of dimension three.
\end{prop}

The aditional orthogonal condition leads the following result. 
\begin{prop} The group of orthogonal automorphisms of $(\mg_{2,3},B_{2,3})$ is 
$$\Auto(\mg_{2,3}) \simeq \mathcal S \ltimes H_1,$$
 where $\mathcal S$ is the subgroup of $\mathcal G$ consisting of the matrices $\tilde{A}$ with $A\in GL(2,K),\,\det(A)=\pm1$. The action in the semidirect product is the restriction of the action of $\mathcal G$ in the Heisenberg Lie group  $H_1$ described before.
\end{prop}

\begin{cor} The Lie algebra of derivations of $\mg_{2,3}$ is isomorphic to 
$$\mathfrak{gl}(2,K) \ltimes (\mathfrak h_1 \times K^3).$$
In particular, fix the ad-invariant metric $\la \,,\,\ra$ in $(\mg_{2,3}, B_{2,3})$ given by the matricial representation as in (\ref{m3223}). 	

The set of skew-symmetric derivations is represented by the Lie algebra
$$\Dera(\mg_{2,3}) \simeq \mathfrak{sl}(2,K) \ltimes \mathfrak h_1$$
while the set of inner derivations is (isomorphic to) $\hh_1$.
\end{cor}

In the next paragraphs we describe $\Der(\mg_{2,3})$ and $\Dera(\mg_{2,3})$ explicitly as subalgebras of $\mathfrak{gl}(5,K)$ putting emphasis on the actions and representations.
\smallskip

By canonical computations one verifies that an element in $\Der(\mg_{2,3})$ has the following  matricial representation in the basis $e_1,e_2,e_3,e_4,e_5$
 
$$\left( \begin{matrix}
 z_1+z_4 &-z_2 & 0 & 0 & 0 \\
-z_3 & z_1 -z_4& 0 & 0 & 0 \\
 z_6& z_5 &2z_1 & 0 & 0 \\
 z_7+z_9 & z_8 & z_5 & 3z_1+z_4 & -z_2\\
 z_{10}& z_7-z_9 & -z_6 & -z_3 & 3z_1-z_4
 \end{matrix}
 \right).
 $$

As usual, denote by $E_{ij}$ the 5$\times$5 matrix which has a 1 at the place $ij$ and $0$ in the other places. 

Consider the vector subspace of dimension three spanned by the matrices
$$
X = E_{32}+E_{43} \qquad Y= E_{31}- E_{53} \qquad Z=E_{41}+ E_{52}
$$
which obey the Lie bracket relation $[X,Y]=Z$, thus this is a faithful representation of the Heisenberg Lie algebra $\hh_1$.

Also denote by
$$U= E_{42}\qquad V=E_{41}-E_{52}\qquad W=E_{51} $$ a basis of 
the vector subspace which spans an abelian Lie algebra of dimension three.

The algebra $\mathfrak h_1 \times K^3$ is an ideal of the Lie algebra of derivations namely the radical. Denote  by $E,F,H, T$ the matrices  given by
$$\begin{array}{rclrcl}
E & = & -E_{12}-E_{45} & F & = & -E_{21}-E_{54} \\
H & = & E_{11}-E_{22} + E_{44}-E_{55} & T & = & E_{11}+E_{22}+2E_{33}+3E_{44}+2E_{55}.
\end{array}
$$

Thus $span \{E,F,H\}\simeq \mathfrak{sl}(2,K)$ and therefore $span \{E,F,H,T\}\simeq \mathfrak{gl}(2,K)=\mathfrak{sl}(2,K)\times K$.
The action of $\mathfrak{gl}(2,K)$ on $\mathfrak h_1\times K^3$ preserves each of the ideals $\mathfrak h_1$ and  $K^3$ respectively.

The action of $\mathfrak{sl}(2,K)$ on $\mathfrak h_1$ is  is given by  derivations, that is $t \in \mathfrak{sl}(2, K) \simeq \Der(\hh_1)$, in the basis $X,Y,Z$, is represented by
$$\left(
\begin{matrix}
h & e & 0\\
f &  -h &  0 \\
0 &  0  & 0
\end{matrix}
\right).
$$

The action of $\mathfrak{sl}(2,K)$ on $K^3$ is given by its irreducible representation in dimension three (see \cite{Hu})
$$
\ad(eE+fF+hH)|_{U,V,W} = \left( \begin{matrix}
2h &  2e &  0\\
f & 0 &  e\\
0 & 2f & -2h
\end{matrix}
\right).
$$

On the other hand the action of $T$ on $\mathfrak h_1$ is diagonal
$$ T \cdot X= X \qquad T \cdot Y = Y \qquad T \cdot Z = 2 Z;$$
and  the action of $T$ on $K^3$ is twice the identity: $T \cdot A= 2 A$ for all $A \in K^3$.

\subsection{The Lie algebra $\mg_{3,2}$} Let $e_1, e_2, e_3, e_4, e_5, e_6$ denote the basis of $\mg_{3,2}$ given in (\ref{b32}) above. A derivation of this Lie algebra has a matrix as follows
$$\left( \begin{matrix}
 z_1+z_{2} & z_4 &  z_6&0 & 0 & 0 \\
 z_{5} & z_1+z_{3} & z_{8}&0 & 0 & 0 \\
  z_{7}& z_9 &z_1 -z_{2}-z_3 &0 & 0 & 0 \\
 -z_{12} & -z_{16}+z_{17} & z_{18} & 2z_1+z_{2}+z_3& z_{8}&-z_6\\
 z_{14}+z_{15}& z_{10} & z_{16}+z_{17} & z_9 & 2z_1-z_3&z_4\\
 z_{13} &z_{14}-z_{15}&z_{11}-z_{12}&-z_{7}&z_{5}&2z_1-z_{2}
 \end{matrix}
 \right).
 $$
 Let $\mathfrak G$ and $\mathfrak R$ denote respectively the  sets of matrices in $\Der(\mg_{3,2})$
$$
\mathfrak G=\left\{ \left( \begin{matrix}
 z_1+z_{2} & z_4 &  z_6&0 & 0 & 0 \\
 z_{5} & z_1+z_{3} & z_{8}&0 & 0 & 0 \\
  z_{7}& z_9 &z_1 -z_{2}-z_3 &0 & 0 & 0 \\
    0    &      0 & 0    & 2z_1+z_{2}+z_3& z_{8}&-z_6\\
     0 &  0 &  0    & z_9 & 2z_1-z_3&z_4\\
    0 &   0  & 0&-z_{7}&z_{5}&2z_1-z_{2}
 \end{matrix}
\right) \right\}
$$
$$ \mathfrak R= \left\{ 
\left( \begin{matrix}
 0 &0    &0    &0 & 0 & 0 \\
 0 &  0    &  0   &0 & 0 & 0 \\
    0  & 0   & 0  &0 & 0 & 0 \\
 -z_{12} & -z_{16}+z_{17} & z_{18} & 0 & 0    &0   \\
 z_{14}+z_{15}& z_{10} & z_{16}+z_{17} & 0   & 0      &0  \\
 z_{13} &z_{14}-z_{15}&z_{11}-z_{12}& 0    & 0   &0     \end{matrix}
\right)\right\}.
$$

With the usual conventions, denote by  $E_{ij}$  the 6$\times$6 matrix which has a $1$ in the file $i$ and column $j$ and $0$ otherwise. 
Let $T$ and $f_i$, $i=1, \hdots , 8$, be the following  matrices

$$ T  =  E_{11}+ E_{22} +E_{33}+2E_{44}+2E_{55}+2E_{66}$$
$$\begin{array}{cc}
\begin{array}{rcl}
f_1 & = & E_{11}-E_{33}+E_{44}-E_{66}\\
f_2 & = & E_{22}-E_{33}+E_{44}-E_{55}\\
f_3 & = & E_{12} + E_{56}\\
f_4 & = & E_{21} + E_{65} 
\end{array}&
\begin{array}{rcl}
f_5 & = & E_{13} -E_{46}\\
f_6 & = & E_{31} - E_{64}\\
f_7 & = & E_{23} + E_{45}\\
f_8 & = & E_{32} + E_{54}
\end{array}
\end{array}
.$$

With the  Lie bracket of matrices, the vector space spanned by $f_1, \hdots, f_8$ constitute a Lie algebra isomorphic to $\mathfrak{sl}(3,K)$, such that  $[T,f_i]=0$ for all $i=1, \hdots , 8$. Hence one has the following isomorphism of Lie algebras

$$\mathfrak G \simeq \mathfrak{sl}(3,K) \times K = \mathfrak{gl}(3,K).$$

For $i=1, \hdots 9$, let $A_i$ denote the matrices 
$$\begin{array}{rclrclrcl}
A_1 & = & E_{52} & A_2 & = & E_{63} & A_3 & = & -E_{41}-E_{63}\\
A_4 & = & E_{61} & A_5 & = & E_{51}+E_{62} & A_6 & = & E_{51}-E_{62}\\
A_7 & = & -E_{42}+ E_{53} & A_8 & = & E_{42}+E_{53} & A_9 & = & E_{43}
\end{array}
$$
which generate the abelian Lie algebra $\mathfrak R$ of dimension nine. Actually this is a faithful representation of minimal dimension of the Lie algebra $K^9$ (see for instance  \cite{Bu}). Hence
$$\mathfrak R \simeq K^9$$
and it coincides with the radical of the Lie algebra of derivations of $\mg_{3,2}$.

The action of $\mathfrak{sl}(3,K)$ on $K^9$ is given by the adjoint representation $f_i \cdot A_j=[f_i, A_j]$ for all $i,j$, while the action of $T$ on $K^9$ is represented by the identity map.

\begin{prop} Let $\mg_{3,2}$ denote the free 2-step nilpotent Lie algebra in three generators. The Lie algebra of derivations of $\mg_{3,2}$ is isomorphic $\mathfrak{gl}(3,K) \ltimes K^9$.
\end{prop}

For  the ad-invariant metric $\la \,,\,\ra$ on $(\mg_{3,2}, B_{3,2})$ as in (\ref{m3223}) one has the next result.

\begin{cor} The set of skew-symmetric derivations of $\mg_{3,2}$ with the  metric $\la \,,\,\ra$ is given by the Lie algebra 
$$\Dera(\mg_{3,2})\simeq \mathfrak{sl}(3,K)\ltimes K^3,$$
while the set of inner derivations is isomorphic to $K^3$. 
\end{cor}

\begin{proof}
One can easily check that $t\in \Dera(\mg_{3,2})$ is skew-symmetric if and only if it belongs to the vector space spanned by
$$\{f_i, i=1, \hdots 8,A_2 +\frac{1}{2} A_3, A_5, A_8\}$$ 

The elements $f_i$ generate a Lie algebra isomorphic to $\mathfrak{sl}(3,K)$ and $A_2 +\frac{1}{2} A_3, A_5, A_8$ span an abelian ideal, isomorphic to $K^3$.
Furthermore $t \in \mathfrak{sl}(3)$ acts on $K^3$ as a linear transformation of $K^3$ (that is $s(x)$ for $x\in K^3$).
\end{proof}

\vskip 5pt

{\bf Acknoledgements.}   The authors are very grateful to A. Kaplan for useful suggestions and comments. 

They also specially thank to an anonymous referee whose  suggestions helped to improve the results in the paper.

\end{document}